\documentclass[11pt]{amsart}
\numberwithin{equation}{section}
\usepackage[english]{babel}
\usepackage[T1]{fontenc}
\usepackage{indentfirst}
\usepackage{enumitem}
\usepackage{amsmath,amssymb, amsbsy}
\usepackage{comment}
\usepackage{amsfonts}
\usepackage{hyperref}
\usepackage{cleveref}
\usepackage{esint}
\usepackage{latexsym}
\usepackage{amsthm}
\usepackage{bm}
\usepackage[dvips]{graphicx}
\usepackage{xcolor}
\usepackage{tikz}
\usepackage{pgfplots}
\usetikzlibrary{intersections}
\usepackage{tikz-3dplot}
\usepackage[outline]{contour}
\DeclareGraphicsExtensions{.pdf,.png,.jpg,.eps}
\usepackage{amsmath}
\usepackage{mathdots}
\usepackage{yhmath}
\usepackage{cancel}
\usepackage{color}
\usepackage{siunitx}
\usepackage{array}
\usepackage{multirow}
\usepackage{amssymb}
\usepackage{textcomp}
\usepackage{gensymb}
\usepackage{tabularx}
\usepackage{extarrows}
\usepackage{booktabs}
\usetikzlibrary{fadings}
\usetikzlibrary{patterns}
\usetikzlibrary{shadows.blur}
\usetikzlibrary{shapes}
\usepackage{bbm}
\usepackage[font=small,labelfont=bf]{caption}
\usepackage[utf8]{inputenc}
\usepackage[active]{srcltx}

\allowdisplaybreaks

\usepackage{doi}

\newcommand{\be}{\begin{equation}}
	\newcommand{\ee}{\end{equation}}

\newcommand{\brd}[1]{\mathbb{#1}}
\newcommand{\R}{\brd{R}}

\newcommand{\N}{\brd{N}}

\newcommand{\pa}{\partial}

\newcommand{\dist}{\operatorname{dist}}

\newcommand{\loc}{{\rm loc}}

\newtheorem{teo}{Theorem}[section]

\newtheorem{Theorem}[teo]{Theorem}

\newtheorem{proposition}[teo]{Proposition}
\newtheorem{lemma}[teo]{Lemma}
\theoremstyle{definition}

\newtheorem{remark}[teo]{Remark}

\pgfplotsset{compat=1.18}


\title[Normalized solutions to the NLS equation in the ball for any prescribed mass]{Normalized solutions to the NLS equation in the ball for any prescribed mass}

\author{Nicola Soave}

\address{Nicola Soave\newline\indent Dipartimento di Matematica ``G. Peano'' \newline\indent Universit\`a degli Studi di Torino \newline\indent Via Carlo Alberto 10, 10123, Torino, Italy} \email{nicola.soave@unito.it}

\begin{document}

\begin{abstract}
Given $\rho>0$, we consider the problem
\[
\text{find $(\lambda,u) \in \R \times H_0^1(B)$ such that } \begin{cases} -\Delta u+\lambda u = |u|^{p-1}u & \text{in } B \\ \int_B u^2\,dx = \rho, \end{cases}
\]
where $B$ is a ball in $\R^N$, $N \ge 1$, and $1<p<2^*-1$. Without any further restriction on $N$, $\rho$ and $p$, we prove the existence of infinitely many radial solutions, and, in dimension $N \ge 4$, of at least one non-radial solution. 
\end{abstract}

\maketitle

\section{Introduction}
In this paper, we study the stationary nonlinear Schr\"odinger equation on a ball \(B \subset \R^N\), with \(N \geq 1\), subject to homogeneous Dirichlet boundary conditions:
\begin{equation}\label{pb ex ball}
\begin{cases}
-\Delta u+\lambda u = |u|^{p-1}u & \text{in } B \\
u=0 & \text{on } \pa B.
\end{cases}
\end{equation}
We seek solutions satisfying the prescribed mass constraint
\be\label{mass constraint}
\int_B u^2\,dx=\rho.
\ee
Since the mass is prescribed, the parameter \(\lambda\) in \eqref{pb ex ball} is itself an unknown. Our main result establishes the existence of infinitely many solutions to \eqref{pb ex ball}--\eqref{mass constraint}, for every \(\rho>0\), throughout the full Sobolev-subcritical range $1<p<2^*-1$, where \(2^*=2N/(N-2)\) if \(N\geq 3\), and \(2^*=+\infty\) if \(N=1,2\).

The main motivation for studying \eqref{pb ex ball} comes from the search for standing waves of the time-dependent nonlinear Schr\"odinger equation
$$
i\psi_t+\Delta\psi+|\psi|^{p-1}\psi=0
\qquad\text{in } \R\times\Omega,
$$
subject to homogeneous Dirichlet boundary conditions on \(\pa\Omega\), where \(\Omega\subset\R^N\). Equations of this type on a bounded domain arise in nonlinear optics in the study of
self-focusing of laser beams propagating in a bounded medium, the Dirichlet condition
modelling the confinement induced by the boundary; we refer to \cite{FiMe} and the references
therein for a discussion of the model. Such solutions have the form $\psi(t,x)=e^{i\lambda t}u(x)$, with \(\lambda\in\R\), and their spatial profile \(u\) solves the corresponding stationary equation.

From a variational perspective, there are two complementary approaches to the construction of standing waves. One may prescribe the \emph{frequency} \(\lambda\) and look for critical points of the associated action functional in \(H_0^1(\Omega)\); in this formulation, however, the mass of the resulting solution is not known a priori. Alternatively, one may prescribe the \emph{mass} \(\rho\) and seek critical points of the energy functional under the corresponding \(L^2\)-constraint, in which case the frequency \(\lambda\) emerges as an unknown Lagrange multiplier. Although these two approaches are closely related (see, for instance, the recent contributions \cite{DeDoGaSe, DoSeTi, JeLu}), it is not yet fully understood how results obtained in one framework can be transferred to the other. A notable exception is the case \(\Omega=\R^N\), where, for the pure-power equation considered here, a simple scaling argument makes the two formulations equivalent. 

The study of prescribed-mass solutions to Schr\"odinger-type equations has become an extremely active area of research over the past few decades. Even a partial overview of the results which have been obtained in this framework would go beyond the scope of the present note; we therefore focus only on the contributions directly related to ours. 

The mass-constrained formulation is particularly natural from the physical point of view. Indeed, besides being conserved by the time-dependent Schr\"odinger flow, the mass often has a direct physical interpretation: it represents, for example, the optical power in nonlinear optics and the total number of particles in a Bose--Einstein condensate, two major areas of application of the NLS. This formulation is equally relevant from a purely mathematical perspective, since it provides valuable information about qualitative properties of stationary solutions, most notably the stability or instability of the associated standing wave.

At the same time, the presence of the \(L^2\)-constraint introduces substantial mathematical difficulties compared with the fixed-frequency problem. These become especially pronounced in the \(L^2\)-supercritical range $1+4/N<p<2^*-1$, where the energy functional is unbounded from below on the \(L^2\)-sphere for every \(\rho>0\). Consequently, solutions can no longer be obtained by direct global minimization, and one is led instead to search for local minimizers or min--max critical points. In this setting, even the construction of a bounded Palais--Smale sequence is a delicate issue. 

In view of these additional difficulties, even some rather basic questions—long settled in the fixed-frequency setting—remain open for the prescribed-mass problem. Arguably, the most elementary unresolved question, at least in its formulation, is the following: 

\smallskip

\begin{center}
\emph{given a bounded domain \(\Omega\subset\R^N\), a prescribed mass \(\rho>0\), and an exponent \(1<p<2^*-1\), does problem \eqref{pb ex ball}--\eqref{mass constraint}, with \(\Omega\) in place of \(B\), admit at least one solution?} 
\end{center}

\smallskip

\noindent The answer is affirmative in the $L^2$-subcritical case $1<p<1+4/N$: for these exponents, the energy is bounded from below on the mass constraint, and existence and multiplicity can be treated with classical variational arguments (Krasnosel'skii genus); see e.g. \cite{PieVer, SongZou}. But for $1+4/N \le p<2^*-1$ the problem is essentially open, and existence and multiplicity are known only for small masses; we refer to  \cite{GaWe, JeSo, NoTaVe, PieVer, SongZou, Wei1, Wei2} for the available results, see also \cite{NoTaVe2, PePiVaVe, PieVerYu} for strictly related results. The difficulty of the question in this range of $p$ is made clear by some of the results obtained in \cite{NoTaVe, PieVer}. In particular, \cite[Theorem 1.2]{PieVer} establishes that, throughout the range $1+4/N \le p<2^*-1$ and for every \(k\in\N\), the set of masses \(\rho>0\) for which \eqref{pb ex ball}--\eqref{mass constraint} admits a solution of Morse index \(k\) is bounded. Consequently, along any sequence of solutions whose masses tend to infinity, the corresponding Morse indices must also diverge. This forced growth of the Morse index, together with the lack of a priori bounds for constrained Palais--Smale sequences and the fact that the energy is unbounded from below on the \(L^2\)-sphere, makes a variational proof of existence for large masses particularly delicate.

To the best of our knowledge, a complete answer to the previous question has so far been obtained only for very specific symmetric domains, such as hyper-rectangles, see \cite[Theorem 1.13]{PieVer}: exploiting the symmetry of the domain, solutions with arbitrary mass can be constructed by gluing together solutions with smaller masses in scaled (smaller) rectangles in a suitable way (the existence of solutions with small mass is also proved in \cite{PieVer}). Moreover, in an annulus, positive radial normalized solutions have been studied in \cite{LiSo} by analyzing the mass along the curve of positive ground states: there, arbitrary masses are attained whenever $N\ge3$, or $N=2$ and $p<5$, the threshold being dictated by the one-dimensional mass-critical exponent, since the solutions concentrate on a sphere as $\lambda\to+\infty$.

Regarding the very natural case of the ball, in \cite{PieVer} the authors adapt the idea successfully developed to treat rectangles in order to obtain solutions in the ball starting from solutions with smaller masses in circular sectors: in this way, they proved the existence of a non-radial solution for every \(\rho>0\) in the range
$
1<p<1+4/(N-1).
$
Note that the exponent \(1+4/(N-1)\) lies above the $L^2$-critical one $1+4/N$, but below the Sobolev-critical exponent
$
2^*-1=1+4/(N-2).
$
Therefore, the case of the ball is still open in full generality. 

The main purpose of the present note is to close this gap and provide a rather complete analysis of the problem. More precisely, for every $N \ge 1$, \(\rho>0\) and every \(1 < p<2^*-1\), we prove that problem \eqref{pb ex ball}--\eqref{mass constraint} admits infinitely many radial solutions; the proof of this result requires a very careful analysis of the behavior of radial solutions of the Schr\"odinger and Lane-Emden equations, an analysis which may be of independent interest. Furthermore, by refining the gluing technique used in \cite[Theorem 1.13]{PieVer}, we also show existence of a non-radial solution to \eqref{pb ex ball}-\eqref{mass constraint} for every \(\rho>0\) and every \(1+4/N \le p<2^*-1\), in any dimension $N \ge 4$.

\begin{Theorem}\label{thm: existence in the ball}
Let $N \ge 1$, $p \in (1,2^*-1)$, and $\rho>0$. Then problem \eqref{pb ex ball}-\eqref{mass constraint} has infinitely many radial solutions, with an increasing number of simple zeros tending to $+\infty$, and associated frequency $\lambda_j<0$. \\
Furthermore, if $N \ge 4$, $p \in [1+4/N, 2^*-1)$ and $\rho>0$, then problem \eqref{pb ex ball}-\eqref{mass constraint} also has at least a non-radial solution.
\end{Theorem}

To prove the theorem, we treat the radial and the non-radial cases separately. The non-radial case, presented in Section \ref{sec: non-rad}, is a refinement of the strategy already developed in the proof of \cite[Theorem 1.13]{PieVer}. We also refer to \cite[Theorem 1.7]{Wei2} for a somewhat similar construction in dimension $N \ge 4$ which gives multiplicity of non-radial solutions, but only for small masses. The radial case is more delicate, and is the content of Section \ref{sec: rad}. A crucial ingredient consists in giving careful estimates of the mass of radial solutions with a given number of zeros to the Lane-Emden equation in the ball. This is the content of Proposition \ref{lem: stima massa} below, which seems to be new and may be of independent interest. Similar estimates, but on the Sobolev norm, were obtained in \cite{Kaj} (the connection between our result and the one in \cite{Kaj} will be discussed in more detail in Remark \ref{rem: on Kaj}). We also mention that we find further information on the associated frequencies $\lambda_j$. This is discussed in Remark \ref{rem on lambda_j}.

\begin{remark}
We conclude by pointing out some natural questions left open by our analysis.

\emph{(i) The Sobolev-critical threshold.} The restriction $p<2^*-1$ enters the proof of the
radial case only through the exponent $b$ used in the Emden-Fowler transformation introduced in Subsection \ref{sub: 0}, which is
positive precisely when $p<2^*-1$; for $p\ge2^*-1$ the method breaks down. The question is
nonetheless meaningful. On the one hand, for $N\ge3$ and $\Omega$ star-shaped, the Pohozaev
identity implies that \eqref{pb ex ball} has no nontrivial solution when $p\ge2^*-1$ and $\lambda\ge0$. On the other hand, all the solutions we construct have $\lambda<0$, and in that regime solutions do
exist: for $p=2^*-1$ and $N\ge4$ this is the content of the classical results of Brezis and
Nirenberg \cite{BrNi}. Whether every mass is attained in the Sobolev-critical and supercritical
regime therefore seems to be an interesting open problem. In the same direction, our
approach only produces solutions with $\lambda<0$; positive frequencies would correspond to
$\gamma<0$ in \eqref{def v} below, in which case $v_\gamma$ need not oscillate and the method gives no
information.

\emph{(ii) Non-radial solutions in low dimension.} For $m=1$ the argument of Section 3 reduces to that of 
 \cite[Theorem 1.13]{PieVer}, and does not provide solutions beyond their threshold $p<1+4/(N-1)$. Consequently, for $N=2$ and $N=3$ the existence of a non-radial solution for every prescribed mass remains open in the range $1+4/(N-1)\le p<2^*-1$, that is, $p\ge5$ for $N=2$ and $3\le p<5$ for $N=3$. 

\emph{(iii) Other domains.} Our proof of the radial case rests on the shooting method, hence
on the radial structure of $B$ and on the ``regular singular point $r=0$", and is not available
even for an annulus. The question raised in the introduction remains open for every bounded
domain which does not admit a tiling of the type exploited in \cite{PieVer}.
\end{remark}

Throughout the paper, $C$ denotes a positive constant which may change from line to line.


\section{Proof of Theorem \ref{thm: existence in the ball} - radial case}\label{sec: rad}

Since we are interested in solutions to \eqref{pb ex ball}-\eqref{mass constraint} for every $\rho>0$, without loss of generality we suppose that $B=B_1$ is the unit ball. The general case can be covered by scaling. To search for radial solutions, we consider the initial value problem
\be\label{def v}
\begin{cases} 
v_\gamma''+\frac{N-1}{r} v_\gamma' +\gamma v_\gamma + |v_\gamma|^{p-1} v_\gamma=0 & \text{for }r \in (0,+\infty)\\
v_\gamma(0) = 1, \qquad v_\gamma'(0)=0,
\end{cases}
\ee
with $\gamma \ge 0$. Standard ODE theory provides a unique local solution. Moreover,
the energy
\be\label{def energy}
  \mathcal E_\gamma(r):=\frac12|v_\gamma'(r)|^2+\frac{\gamma}{2}|v_\gamma(r)|^2
  +\frac{1}{p+1}|v_\gamma(r)|^{p+1}
\ee
satisfies 
\[
\mathcal E_\gamma'(r)=-\frac{N-1}{r}|v_\gamma'(r)|^2\le0 \quad \implies \quad \mathcal E_\gamma(r)\le\mathcal E_\gamma(0)=\frac\gamma2+\frac1{p+1} \quad \forall r >0;
\]
in particular $v_\gamma$ and $v_\gamma'$ are bounded, so that the solution extends to the
whole of $(0,+\infty)$, and $|v_0|\le1$ on $[0,+\infty)$. All the zeros of $v_\gamma$ in
$(0,+\infty)$ are simple: if $v_\gamma(r_0)=v_\gamma'(r_0)=0$ for some $r_0>0$, then
$v_\gamma\equiv0$ by uniqueness, contradicting $v_\gamma(0)=1$. Hence the zeros are
isolated, and each bounded interval contains finitely many of them.

Furthermore, for every $p\in(1,2^*-1)$ the function $v_\gamma$ has infinitely many zeros,
which can therefore be ordered in a sequence tending to $+\infty$. For $\gamma>0$ this
follows from Sturm oscillation theory: the function $w:=r^{\frac{N-1}{2}}v_\gamma$, which
has the same zeros as $v_\gamma$ in $(0,+\infty)$, solves
\[
  w''+\Big(\gamma+|v_\gamma|^{p-1}-\frac{(N-1)(N-3)}{4r^2}\Big)w=0 ,
\]
and the coefficient is not smaller than $\gamma/2$ for $r$ large; comparison with
$\zeta''+\frac{\gamma}{2}\zeta=0$ shows that $w$ vanishes in every interval of length
$\pi\sqrt{2/\gamma}$ contained in a neighbourhood of $+\infty$. For $\gamma=0$ the
oscillatory character of $v_0$ is classical, see \cite[Corollary 6.7]{NiNa} and \cite{Kaj}.

We denote by $R_j(\gamma)>0$ the $j$-th zero of $v_\gamma$: since $v_\gamma$ depends
continuously on the parameter $\gamma$ in $C^1_{\mathrm{loc}}([0,+\infty))$ and, as
observed above, $v_\gamma'(R_j(\gamma))\neq0$, the implicit function theorem ensures
that $R_j$ is continuous (in fact $C^1$) on the whole of $[0,+\infty)$, the endpoint
$\gamma=0$ included.

Next, we define
\[
u_{j,\gamma}(r) := R_j(\gamma)^\frac{2}{p-1} v_\gamma(R_j(\gamma)r).
\]
It is easy to check that $u_{j,\gamma}$ is a radial solution to \eqref{pb ex ball} with $\lambda_{j,\gamma}=-\gamma R_j(\gamma)^2$, with exactly $j$ nodal regions (i.e. exactly $j-1$ zeros in $(0,1)$). The mass of $u_{j,\gamma}$ is 
\be\label{mass gamma}
M_j(\gamma):= \int_{B_1} u_{j,\gamma}^2\,dx = |\mathbb{S}^{N-1}| R_j(\gamma)^\frac{4}{p-1} \int_0^1 v_\gamma^2(R_j(\gamma)r)r^{N-1}\,dr,
\ee
and depends continuously on $\gamma \in [0,+\infty)$ as well. The idea of the proof is simple: we aim to study the behavior of $M_j(\gamma)$ at $\gamma =0$ and as $\gamma \to +\infty$. We claim that 
\be\label{cl M infty}
M_j(\gamma) \to 0^+\qquad \text{as $\gamma \to +\infty$, for every $j$},
\ee
and that 
\be\label{cl M 0}
M_j(0) \to +\infty \qquad \text{as $j \to +\infty$}
\ee
(with a precise rate). Let then $\rho>0$ be arbitrarily chosen. By \eqref{cl M 0}, there exists $\bar j \in \N$ such $M_j(0) > \rho$ for every $j > \bar j$. Since moreover $M_j(\gamma) \to 0^+$ as $\gamma \to +\infty$, by continuity there exists $\bar \gamma \in (0,+\infty)$ such that $M_j(\bar \gamma)=\rho$, namely $(\lambda_{j, \bar \gamma}, u_{j,\bar \gamma})$ is a radial solution to \eqref{pb ex ball}-\eqref{mass constraint} with exactly $j-1$ simple zeros; this gives the conclusion of Theorem \ref{thm: existence in the ball} in the radial case.

The rest of the section is devoted to the proof of \eqref{cl M infty} and \eqref{cl M 0}. As we shall see, the proof of \eqref{cl M infty} is rather simple, while the proof of \eqref{cl M 0} is more delicate. 
\begin{remark}\label{rem: on Kaj}
The case $\gamma=0$ corresponds to the Lane-Emden equation. Radial solutions of the Lane-Emden equation have been widely studied, and
several fine asymptotic estimates are available; we refer for instance to
\cite{DeMaIaPa1, DeMaIaPa2, Kaj} and references therein. In particular, for the solutions $u_{j,0}$ considered here, \cite{Kaj}
provides the sharp two-sided estimate
$\|\nabla u_{j,0}\|_{L^2}^2\asymp j^{2(p+1)/(p-1)}$ (see \cite[Example, p.~264]{Kaj}).
This does \emph{not} yield the information on the mass which we need: by the
Nehari identity $\|\nabla u_{j,0}\|_{L^2}^2=\|u_{j,0}\|_{L^{p+1}}^{p+1}$ and the
Gagliardo--Nirenberg inequality one only gets
\[
\|u_{j,0}\|_{L^2}^{(1-\theta)(p+1)} \ge c\,\|\nabla u_{j,0}\|_{L^2}^{\,2-\frac{N(p-1)}{2}},
\qquad \theta=\frac{N(p-1)}{2(p+1)},
\]
whose right hand side is bounded away from $0$ only for $p<1+4/N$, and is
infinitesimal as $j\to\infty$ when $p>1+4/N$. Thus, precisely in the
$L^2$-critical and supercritical regime --- the range of interest here --- the
known estimates give no lower bound on the mass, and a direct analysis is
needed. 
\end{remark}

\subsection{Asymptotics of $M_j(\gamma)$ as $\gamma \to +\infty$}
Let $z_\gamma(r) = v_\gamma(r/\sqrt{\gamma})$, which satisfies
\be\label{eq z}
z_\gamma''+\frac{N-1}{r}z_\gamma' + z_\gamma + \frac{1}{\gamma} |z_\gamma|^{p-1}z_\gamma  =0 \qquad \text{for $r \in (0,+\infty)$}.
\ee
We observe that the energy
\[
E_\gamma(r) := \frac{1}{\gamma} \mathcal{E}_\gamma\left( \frac{r}{\sqrt{\gamma}}\right), 
\]
with $\mathcal{E}_\gamma$ defined in \eqref{def energy}, is a Lyapunov functional, and
\be\label{uniform bound z}
E_\gamma(r) \le E_\gamma(0) = \frac12 + \frac{1}{\gamma(p+1)} \quad \implies \quad |z_\gamma(r)|+|z_\gamma'(r)| \le C \quad \forall r \in [0,+\infty),
\ee
with $C>0$ independent of $\gamma >1$. These upper bounds, together with Eq. \eqref{eq z}, ensure that up to a subsequence $z_\gamma \to z$ in $C^2_{\loc}([0,+\infty))$, where $z$ solves the linear problem
\[
\begin{cases} 
z''+\frac{N-1}{r} z' +z =0 & \text{for }r \in (0,+\infty)\\
z(0) = 1, \qquad z'(0)=0.
\end{cases}
\]
Actually, since the solution of this problem is unique, the convergence takes place not only up to a subsequence, but along the full family $\{z_\gamma\}$. Classical Sturm oscillation theory ensures that $z$ has infinitely many simple zeros $0<\zeta_1<\zeta_2<\dots$ (note that $z(r)=c_N r^{-(N-2)/2}J_{(N-2)/2}(r)$, where $J_{(N-2)/2}$ is a Bessel function). Now, clearly the zeros $\bar R_j(\gamma)$ of $z_\gamma$ are determined by those of $v_\gamma$ via the relation $\bar R_j(\gamma) = \sqrt{\gamma} R_j(\gamma)$, and by $C^1$ convergence we have $\bar R_j(\gamma) \to \zeta_j$ as $\gamma \to +\infty$. Since $\zeta_j>0$, we infer that $R_j(\gamma) \to 0^+$, and hence, recalling \eqref{mass gamma} and \eqref{uniform bound z}, it follows that
\[
M_j(\gamma) \le C_N R_j(\gamma)^\frac{4}{p-1} \|v_\gamma\|_{L^\infty([0,+\infty))}^2 = C_N R_j(\gamma)^\frac{4}{p-1} \|z_\gamma\|_{L^\infty([0,+\infty))}^2 \to 0^+
\]
as $\gamma \to +\infty$, which proves claim \eqref{cl M infty}.

\subsection{Behavior of $M_j(0)$}\label{sub: 0} Here we analyze the case $\gamma=0$, which corresponds to the Lane-Emden equation, and turns out to be much more involved. 

Let $v=v_0$ defined in \eqref{def v}. We consider the following transformation:
\[
t(r)=\frac{r^b}{b} \quad \text{and} \quad v(r) = r^{-a}y(t(r)),
\]
with $a,b>0$ to be chosen appropriately. Denoting by $\dot y$ and $\ddot y$ the derivatives of $y$ with respect to $t$, it is not difficult to check that 
\[
v''+\frac{N-1}{r} v' = r^{2b-a-2} \ddot y + (b+N-2a-2) r^{b-a-2} \dot y+a(a+2-N) r^{-a-2} y.
\] 
On the other hand, the left hand side is also equal to
\[
-|v|^{p-1} v = -r^{-ap}|y|^{p-1}y.
\]
Thus, we obtain a differential equation for $y$. We choose $a$ and $b$ in such a way that the coefficient of the first order term vanishes, and that the power term in front of $\ddot y$ coincides with the one in front of $|y|^{p-1}y$: namely, we impose the conditions
\[
\begin{cases}
b+N-2a-2=0\\
2b-a-2 = -ap
\end{cases} \quad \iff \quad a= \frac{2(N-1)}{p+3}, \quad b= \frac{N+2-(N-2)p}{p+3}.
\]
In this way, the differential equation for $y$ can be written as
\[
\ddot y + |y|^{p-1} y + \frac{\kappa}{t^2} y=0 \qquad \text{for $t \in (0,+\infty)$, where }\kappa:=\frac{a(a+2-N)}{b^2}.
\]
This suggests that, for large times, $y$ behaves like a solution of the autonomous problem $\ddot Y+ |Y|^{p-1}Y=0$, which has infinitely many periodic solutions whose nodal behavior can be conveniently controlled (of course, this heuristic has to be made rigorous). We wish to exploit this information, in order to obtain the desired estimates on $M_j(0)$. In what follows, we denote by $r_j=R_j(0)$ and $t_j$ the simple zeros of $v$ and $y$, respectively; clearly, $t_j=r_j^b/b$, and both the sequences diverge as $j \to +\infty$.

\medskip

\noindent\textbf{Step 1) Study of the energy.} We consider both the classical and a modified energy associated with $y$, defined by
\[
H(t):= \frac12|\dot y(t)|^2 +\frac{1}{p+1}|y(t)|^{p+1}, \quad K(t):= H(t) + \frac{\kappa}{2t^2} y^2(t).
\]
Moreover, for any $c \ge 0$ we denote by $\Gamma_c$ the level set 
\[
\Gamma_c:= \{\mathcal{H}(q,v)=c\} \subset \R^2, \quad \text{where} \quad
\mathcal{H}(q,v) = \frac{1}{2}|v|^2+\frac1{p+1}|q|^{p+1}
\]
is the Hamiltonian associated with the autonomous nonlinear oscillator. Note that $\Gamma_c$ is a closed regular curve for every $c>0$.
\begin{lemma}\label{lem: studio energia}
There exist $C, \ell>0$ such that:
\begin{itemize}
\item[($i$)] $|y(t)|+|\dot y(t)| \le C$ for every $t \ge 1$, and $H(t), K(t) \to \ell$ as $t \to +\infty$;
\item[($ii$)] $\dist((y(t), \dot y(t)), \Gamma_\ell) \to 0$ as $t \to +\infty$;
\item[($iii$)] for all sufficiently large $t$
\be\label{rate K}
|H(t)-\ell| \le \frac{C}{t^2}.
\ee
\end{itemize}
\end{lemma}
\begin{proof}
We compute
\be\label{der K}
\dot K(t) = - \frac{\kappa y^2(t)}{t^3},
\ee
and we distinguish three cases for the proof of point ($i$).

\emph{Case 1: $N \ge 3$.} For $N \ge 3$, the constant $\kappa$ is negative for every $p>1$. Therefore, $K$ is increasing. We claim that it is also bounded for $t>1$. 

To prove the claim, we note that by the Young inequality, for every $t >1$
\[
|\dot H(t)| = \frac{|\kappa|}{t^2}|y||\dot y|  \le \frac{|\kappa|}{t^2}\left(\frac12|\dot y|^2 + \frac{2}{p+1}|y|^{p+1} + C_p\right) \le \frac{C}{t^2}(H(t)+1),
\]
for some $C>0$ depending on $p$ and $N$. Therefore
\[
\frac{\dot H(t)}{H(t)+1} \le \frac{C}{t^2} \quad \implies \quad H(t) +1 \le C e^{C} \quad \forall t >1,
\]
namely $H$ remains bounded for $t>1$; this ensures that $y$ and $\dot y$ are bounded as well, and hence $\dot K(t) \le C/t^3$ for every $t>1$. By integrating again, we infer that also $K(t)$ remains bounded as $t \to +\infty$, which in turn implies that both $K$ and $H$ tend to the same constant $\ell \in [0,+\infty)$ as $t \to +\infty$. It remains to show that $\ell>0$. Suppose by contradiction that $\ell = 0$. Recalling that $t_j$ denotes a zero of $y$, we have 
\[
0 \le \frac12 |\dot y(t_j)|^2 = K(t_j) \le \lim_{t \to +\infty} K(t) = 0,
\]
where we used the monotonicity of $K$. Therefore $y(t_j)=\dot y(t_j) = 0$, and by the uniqueness theorem for ODEs we infer that $y \equiv 0$, a contradiction.

\emph{Case 2: $N = 2$.} The first observation is that $\kappa=1/4>0$, so that by \eqref{der K} the function $K$ is strictly positive and decreasing. Note that, by definition, this immediately implies that $|y|$ and $|\dot y|$ stay bounded on $[1,+\infty)$. Thus, it remains to show that $K$ (and hence also $H$) does not tend to $0$ as $t \to +\infty$. Suppose by contradiction that $K(t) \to 0$ as $t \to +\infty$. On the one hand, we claim that 
\be\label{n=2 cl1}
H(t_j) \le C t_j^{-\frac{2(p+1)}{p-1}}.
\ee
To prove this claim, we note that by \eqref{der K} and the assumption $K(\infty)=0$ 
\be\label{exp K}
K(t) = \kappa\int_{t}^{+\infty} \frac{y^2(s)}{s^3}\,ds.
\ee
Since $t_j$ is a zero of $y$, we have $H(t_j) = K(t_j)$, and hence, by monotonicity,
\[
H(t) = K(t)-\frac{\kappa}{2t^2}y^2(t) \le K(t) \le K(t_j) = H(t_j)
\]
for every $t \ge t_j$. In particular, $|y(t)|^{p+1} \le (p+1)H(t_j)$ for any such $t$, which using \eqref{exp K} implies
\[
H(t_j) =K(t_j) \le C H(t_j)^{\frac{2}{p+1}} t_j^{-2},
\]
whence \eqref{n=2 cl1} follows. On the other hand, we show now that 
\be\label{n=2 cl2}
H(t_j) \ge \frac{C}{t_j},
\ee
which gives a contradiction with \eqref{n=2 cl1} for $j$ sufficiently large. To prove this estimate, we use the ``old" function and variable $v(r)$, and we consider
\[
P(r) = r^2\left( \frac{1}{2}|v'(r)|^2 + \frac{1}{p+1}|v(r)|^{p+1}\right).
\]
By using the fact that $N=2$, a simple computation yields 
\[
P'(r) = \frac{2r}{p+1}|v(r)|^{p+1}>0,
\]
namely $P$ is positive and strictly increasing. Furthermore, at each zero $r_j$ of $v$, recalling the relation between $r_j$ and $t_j$ we have
\[
P(r_j) = \frac12 r_j^{\frac{4}{p+3}} |\dot y(t_j)|^2 = b t_j H(t_j).
\]
Hence, by monotonicity and positivity of $P$, we deduce that $b t_j H(t_j) \ge C>0$, whence \eqref{n=2 cl2}, and hence also the convergence $H(t), K(t) \to \ell>0$, follow.

\emph{Case 3: $N=1$.} In this case $a=\kappa=0$. Therefore, the problem is autonomous and $y$ is a periodic (non-trivial) solution. Thus, the conclusion $H(t), K(t) \to \ell>0$ follows trivially.

The previous discussion proves the validity of point ($i$) of the lemma, in any dimension. Point ($ii$) is a simple consequence of point ($i$) and of the compactness of $\Gamma_\ell$. It remains to prove \eqref{rate K}. By direct computations
\[
\dot H(t) = -\frac{\kappa}{t^2} y(t) \dot y(t) = -\frac{\kappa}{2t^2}\frac{d}{dt} y(t)^2,
\]
whence
\[
\ell - H(t) = \frac{\kappa}{2t^2} y(t)^2-\kappa \int_t^{+\infty} \frac{y(s)^2}{s^3}\,ds.
\]
In view of the boundedness of $y$, this gives the desired result.
\end{proof}

\noindent\textbf{Step 2) Convergence to a periodic solution of the autonomous nonlinear oscillator.}
We next prove that \(y\) converges, together with its derivative, to a suitable
time translation of a periodic solution of the limiting autonomous equation.

\begin{lemma}\label{lem: convergence}
There exists a periodic solution $Y_\infty$ of the equation $\ddot Y_\infty+|Y_\infty|^{p-1}Y_\infty=0$ and a real number $\tau_\infty \in \R$ such that
\[
|y(t)-Y_\infty(t+\tau_\infty)|
+
|\dot y(t)-\dot Y_\infty(t+\tau_\infty)|
=
O(t^{-1})
\qquad\text{as }t\to+\infty.
\]
\end{lemma}

\begin{remark}
Lemma \ref{lem: convergence} is in the spirit of the classical Fowler asymptotics for Emden-Fowler equations. We
point out that the general results on asymptotically autonomous systems (see e.g. \cite{Mar, Thi}) only ensure that the limit set of the trajectory is a periodic orbit of the limiting autonomous system, which is the content of Lemma \ref{lem: studio energia}-($ii$); they do not provide the asymptotic
phase, nor the rate, which are essential in the proof of Proposition \ref{lem: stima massa} below. This is due to the fact that the unperturbed system has a whole period annulus of non-hyperbolic periodic orbits, so
that the classical asymptotic equivalence theorems do not apply. Since we could not find a
reference in the form we need, we give a self-contained proof.
\end{remark}

For the proof, it is convenient to introduce energy-angle variables in the following rather standard way (clearly, these are closely related to the action-angle variables commonly employed in Hamiltonian dynamics). Recall that $\Gamma_h$ is a closed periodic orbit, for every $h>0$. Its amplitude is
\[
A(h)=((p+1)h)^{1/(p+1)}.
\]
Let \(Y_h\) be the solution of $\ddot Y+|Y|^{p-1}Y=0$ with $Y_h(0) = A(h)$, $\dot Y_h(0) = 0$. We denote its period by \(T(h)\) and set
\[
\omega(h):=\frac{2\pi}{T(h)}.
\]
Note that $Y_h(t)=h^{\frac1{p+1}}Y_1\!\big(h^{\frac{p-1}{2(p+1)}}t\big)$, whence $T(h)=T(1)\,h^{-\frac{p-1}{2(p+1)}}$ and $\omega\in C^\infty((0,+\infty))$. We can parametrize the periodic orbit \(\Gamma_h\) by defining
\[
\Phi(h,\theta)
:= X_h\left(\frac{\theta}{\omega(h)}\right), \quad \text{with} \quad 
X_h := (Y_h,
\dot Y_h),
\qquad
\theta\in\mathbb R/(2\pi\mathbb Z).
\]
Thus, \(\theta\) is the time required to move along the nonlinear orbit, normalized so
that one complete period corresponds to an increase of \(2\pi\), counted in the counterclockwise sense. 

Fix \(0<h_-<h_+\) and consider the annular region
\[
\mathcal A
:=
\left\{
(q,v)\in\mathbb R^2:
h_-<H(q,v)<h_+
\right\}.
\]
Then $\Phi:(h_-,h_+)\times\mathbb S^1\longrightarrow\mathcal A$ is a \(C^1\)-diffeomorphism. Indeed, every point of \(\mathcal A\)
belongs to exactly one energy level \(\Gamma_h\), and every point of
\(\Gamma_h\) is attained exactly once by \(Y_h\) during one period. 
If
\[
F(q,v):=\left(v,-|q|^{p-1}q\right)
\]
denotes the vector field of the autonomous oscillator, then the
definition of \(\Phi\) gives
\[
\partial_\theta\Phi(h,\theta)
=
\frac{1}{\omega(h)}F(\Phi(h,\theta)) \quad \iff \quad F(\Phi(h,\theta))
=
\omega(h)\partial_\theta\Phi(h,\theta).
\]
%
%
%
%
%
\begin{proof}[Proof of Lemma \ref{lem: convergence}]
Setting $X(t):=(y(t),\dot y(t))$, the equation of $y$ can be written as
\[
\dot X(t)=F(X(t))+G(t,X(t)), \quad\text{where} \quad
G(t,q,v)
=
\left(0,-\frac{\kappa}{t^2}q\right).
\]
We know that $H(t) \to \ell>0$ as $t \to +\infty$. Hence, for all sufficiently large \(t\), the trajectory \(X(t)\)
remains in a compact annular region surrounding $\Gamma_\ell$, that does not contain the origin.
We may therefore write uniquely $X(t)=\Phi(H(t),\theta(t))$, where \(\theta(t)\) is initially defined modulo \(2\pi\). Since
\([t_0,+\infty)\) is an interval, \(\theta\) admits a continuous
lifting to a real-valued function. 

Now, let \(\Theta\) denote the angular component of the inverse map
\(\Phi^{-1}\). Differentiating the identity $\Theta(\Phi(h,\theta))=\theta$ with respect to $\theta$, and
using that $F(\Phi(h,\theta))=\omega(h)\,\partial_\theta\Phi(h,\theta)$, we obtain \\
$D\Theta(X(t))F(X(t))=\omega(h(t))$, so that
\[
\dot \theta(t)
=
D\Theta(X(t))\dot X(t) =
= \omega(h(t))
+
D\Theta(X(t))G(t,X(t)).
\]
%
%
%
%
The derivative \(D\Theta\) is bounded on the compact annular region
containing the trajectory. Moreover, \(y\) is bounded and
\[
G(t,X(t))
=
\left(0,-\frac{\kappa}{t^2}y(t)\right)
=
O(t^{-2}).
\]
Therefore,
\[
\dot \theta(t)=\omega(H(t))+O(t^{-2}),
\]
and we have already proved that
\[
H(t)-\ell=O(t^{-2}).
\]
Since \(\omega\) is \(C^1\) in a neighborhood of \(\ell\), this yields
\[
\omega(H(t))
=
\omega(\ell)+O(t^{-2}).
\]
Setting $\omega_\infty:=\omega(\ell)$, we conclude that
\[
\dot \theta(t)=\omega_\infty+O(t^{-2}).
\]
Since \(t^{-2}\) is integrable at infinity, there exists
\(\theta_\infty\in\mathbb R\) such that
\[
\theta(t)
=
\omega_\infty t+\theta_\infty+O(t^{-1}).
\]
At this point, it follows from $X(t)=\Phi(H(t),\theta(t))$ and from the regularity of \(\Phi\) that
\[
X(t)
-
\Phi(\ell,\omega_\infty t+\theta_\infty)
= O(t^{-1}) \qquad \text{as $t \to +\infty$}.
\]
Setting $Y_\infty:=Y_\ell$ and $\tau_\infty:=\theta_\infty/\omega_\infty$, the conclusion follows.
\end{proof}

\noindent \textbf{Step 3) Asymptotic estimates for $M_j(0)$.} We are ready to prove the asymptotic estimate \eqref{cl M 0}.

\begin{proposition}\label{lem: stima massa}
There exist constants $0<c_1\le c_2$ such that
\[
c_1\,j^{\frac{4}{p-1}} \le M_j(0) \le c_2\,j^{\frac{4}{p-1}}
\qquad\text{for every }j\in\mathbb N .
\]
In particular, $M_j(0)\to+\infty$ as $j\to\infty$.
\end{proposition}
\begin{proof}
In order to simplify the notation, we write $Y(t):= Y_\infty(t+\tau_\infty)$, given by Lemma \ref{lem: convergence}.
We denote by $\mathcal T>0$ the minimal period of $Y$ and by
\[
A:=\big((p+1)\ell\big)^{\frac{1}{p+1}}>0
\]
its amplitude, so that $\max Y=A$ and $\min Y=-A$. We first record two
elementary facts.

\smallskip
\emph{(a) $y$ is bounded on $(0,+\infty)$.} We have already observed at the beginning of the section that $|v_0| \le 1$ for every $r \ge 0$. Since
$|y(t)|=r^{a}|v(r)|\le (bt)^{a/b}$ for every $t>0$, we deduce that $y$ is bounded on
$(0,1]$; together with Lemma 2.1-(i) this gives
$\|y\|_{L^\infty(0,+\infty)}\le C$.

\smallskip
\emph{(b) For all sufficiently large $t$, every interval of length $\mathcal T$
contains a zero of $y$.} By Lemma 2.2 there exists $T_1>0$ such that
$|y(t)-Y(t)|<A/2$ for every $t\ge T_1$. Let $\theta\ge T_1$. Since
$[\theta,\theta+\mathcal T]$ has the length of one period, $Y$ attains there both
the value $A$ and the value $-A$: let $s_1,s_2\in[\theta,\theta+\mathcal T]$ be
such that $Y(s_1)=A$ and $Y(s_2)=-A$. Then
\[
y(s_1)>\frac A2>0>-\frac A2>y(s_2),
\]
and hence $y$ vanishes at some point strictly between $s_1$ and $s_2$; in
particular $y$ has a zero in $(\theta,\theta+\mathcal T)$. Choosing
$\theta=t_k$, we conclude that
\[
t_{k+1}-t_k<\mathcal T \qquad\text{whenever } t_k\ge T_1 .
\]

Now, since \(Y\) is continuous, periodic, and nontrivial, there exist an
interval \(I\subset[0,\mathcal T]\), with length \(\lvert I\rvert>0\),
and a constant \(\eta>0\) such that
\[
Y(t)^2\geq 4\eta
\qquad\text{for every }t\in I.
\]
By the uniform convergence of \(y(t)-Y(t)\) to zero for large times, there
exists \(T_0>0\) such that
\[
y(t)^2\geq\eta
\qquad\text{for every }t\in I+k\mathcal T
\]
whenever \(I+k\mathcal T\subset[T_0,+\infty)\).

Define
\[
q:=\frac{2-2b}{b} \ge 0
\]
(with strict inequality for every $N \ge 2$, while $q=0$ if $N=1$). For \(S>0\), let
\[
\mathcal N(S)
:=
\#\bigl\{k\in\mathbb N:
I+k\mathcal T\subset[S/2,S]\bigr\}.
\]
Since the interval \([S/2,S]\) has length \(S/2\), we have
\[
\mathcal N(S)\geq \frac{S}{2\mathcal T}-C \quad \implies \quad \mathcal N(S)\geq cS
\]
for some $c>0$ and all sufficiently large \(S\). Therefore,
\begin{equation}\label{est S}
\int_0^S t^q y(t)^2\,dt\ \ge \sum_{I+k\mathcal T\subset[S/2,S]}
\int_{I+k\mathcal T} t^q y(t)^2\,dt\ \ge\ \eta\left(\frac S2\right)^{q}|I|\,\mathcal N(S)
\ \ge\ c_1 S^{q+1}
\end{equation}
for every sufficiently large $S$. On the other hand, since $q\ge0$, fact (a)
gives
\begin{equation}\label{est S2}
\int_0^S t^q y(t)^2\,dt\ \le\ \|y\|_{L^\infty(0,+\infty)}^2\int_0^S t^q\,dt
\ =\ \frac{\|y\|^2_{L^\infty(0,+\infty)}}{q+1}\,S^{q+1}\ \le\ c_2S^{q+1}
\end{equation}
for every $S>0$.

We now return to the radial variable $r$ and the original function $v=v_0$. Recalling that $t=r^b/b$, $dt=r^{b-1}\,dr$, and \(2a=b+N-2\), we obtain
\[
\int_0^{r_j}v(r)^2r^{N-1}\,dr=
\int_0^{r_j}
y\left(\frac{r^b}{b}\right)^2r^{N-1-2a}\,dr =
b^q\int_0^{t_j}t^qy(t)^2\,dt.
\]
Consequently, by \eqref{est S} and \eqref{est S2} and since
\[
t_j^{\,q+1}=b^{-(q+1)}r_j^{\,b(q+1)}=b^{-(q+1)}r_j^{\,N-2a},
\]
we deduce that
\begin{equation}\label{791}
c_1\,r_j^{\,N-2a}\ \le\ \int_0^{r_j} v(r)^2 r^{N-1}\,dr\ \le\ c_2\,r_j^{\,N-2a},
\end{equation}
where we used the fact that $b(q+1)=2-b=N-2a$. It remains to estimate \(r_j\) in terms of $j$. Since \(y\) and \(\dot y\) are
bounded (Lemma \ref{lem: studio energia}-($i$)), \(y\) satisfies
\[
\ddot y+c(t)y=0,
\quad \text{where} \quad
c(t):=|y(t)|^{p-1}+\frac{\kappa}{t^2},
\]
and $|c(t)|\leq\bar C$ for all sufficiently large \(t\), for some \(\bar C>0\). On every
nodal interval \((t_k,t_{k+1})\) contained in this region, multiplying
the equation by \(y\) and integrating by parts gives
\[
\int_{t_k}^{t_{k+1}}|\dot y(t)|^2\,dt
=
\int_{t_k}^{t_{k+1}}c(t)y(t)^2\,dt
\leq
\bar C\int_{t_k}^{t_{k+1}}y(t)^2\,dt.
\]
On the other hand, the Poincar\'e inequality yields
\[
\int_{t_k}^{t_{k+1}}y(t)^2\,dt
\leq
\left(\frac{t_{k+1}-t_k}{\pi}\right)^2
\int_{t_k}^{t_{k+1}}|\dot y(t)|^2\,dt.
\]
Since \(y\) is nontrivial on \((t_k,t_{k+1})\), it follows that
\[
t_{k+1}-t_k\geq\frac{\pi}{\sqrt{\bar C}}.
\]
Summing over the nodal intervals, we obtain $t_j\ge cj$ for all sufficiently
large $j$, for some $c>0$. On the other hand, by fact (b) we have
$t_{k+1}-t_k<\mathcal T$ for every $k$ such that $t_k\ge T_1$, and therefore also
$t_j\le Cj$ for all sufficiently large $j$. Recalling that $r_j=(bt_j)^{1/b}$, we
conclude that
\[
c\,j^{1/b}\ \le\ r_j\ \le\ C\,j^{1/b}
\qquad\text{for all sufficiently large } j .
\]
These estimates, together with \eqref{mass gamma} and \eqref{791}, finally allow us to conclude. Changing variables in \eqref{mass gamma}, we have
\[
M_j(0)=|\mathbb S^{N-1}|\,r_j^{\frac{4}{p-1}-N}\int_0^{r_j}v(r)^2r^{N-1}\,dr ,
\]
so that, by \eqref{791},
\[
c_1\,r_j^{\frac{4}{p-1}-2a} \le M_j(0)\le c_2\,r_j^{\frac{4}{p-1}-2a} .
\]
Since
\[
\frac{4}{p-1}-2a=\frac{4b}{p-1}>0 ,
\]
the two-sided estimate on $r_j$ gives the desired result.
\end{proof}

\begin{proof}[Conclusion of the proof of Theorem \ref{thm: existence in the ball} in the radial case]
As already observed, having established \eqref{cl M infty} and \eqref{cl M 0}, the theorem follows directly.
\end{proof}

\begin{remark}\label{rem on lambda_j}
The frequencies of the solutions given by Theorem \ref{thm: existence in the ball} are localized by the radial
Dirichlet spectrum of the ball. Indeed $\lambda_{j,\gamma}=-\gamma R_j(\gamma)^2
=-\bar R_j(\gamma)^2$, where $\bar R_j(\gamma)=\sqrt\gamma R_j(\gamma)$ is the $j$-th zero
of $z_\gamma$; since $z_\gamma$ solves
\[
  \big(r^{N-1}z_\gamma'\big)'+r^{N-1}\Big(1+\tfrac1\gamma|z_\gamma|^{p-1}\Big)z_\gamma=0
\]
and $\|z_\gamma\|_{L^\infty}^2\le1+\frac{2}{\gamma(p+1)}$ by \eqref{uniform bound z}, the potential lies
between $1$ and $1+C_p/\gamma$, with $C_p:=\big(1+\frac{2}{p+1}\big)^{\frac{p-1}{2}}$.
Comparing with the corresponding constant-coefficient equations, whose solutions bounded
at the origin are $z(r)$ and $z(\sqrt{1+C_p/\gamma}\,r)$, the Sturm comparison theorem
gives
\[
  \frac{\zeta_j}{\sqrt{1+C_p/\gamma}}\ \le\ \bar R_j(\gamma)\ <\ \zeta_j
  \qquad\text{for every } \gamma\ge1 ,
\]
with $C_p$ independent of $j$. Since $\zeta_j^2$ is the $j$-th radial Dirichlet eigenvalue
of $B_1$, we deduce on the one hand that $-\zeta_j^2<\lambda_{j,\gamma}<0$ for every
$\gamma>0$, and on the other hand that $|\lambda_{j,\gamma}|=\zeta_j^2\big(1+O(\gamma^{-1})\big)$
uniformly in $j$. Recalling that $\zeta_j=\pi j+O(1)$ as $j\to\infty$, it follows that
whenever the parameters $\bar\gamma_j$ in the proof of Theorem 1.1 can be chosen so that
$\bar\gamma_j\to+\infty$, the associated frequencies satisfy
$\lambda_j=-\pi^2j^2\big(1+o(1)\big)\to-\infty$.
\end{remark}

\section{Proof of Theorem \ref{thm: existence in the ball} - non-radial case}\label{sec: non-rad}

This is a rather direct extension of \cite[Theorem 1.13]{PieVer}. Let
\[
m:= \left\lfloor \frac{N}{2} \right\rfloor \ge 2,
\]
where $\lfloor \cdot \rfloor$ denotes the integer part. We consider the first \(2m\) variables and group them into pairs: $(x_1,x_2)$, \ldots, $(x_{2m-1},x_{2m})$. Since $N \ge 4$, we have at least two pairs of variables. The idea is to adapt the construction in \cite{PieVer} to each pair (and not only to the first pair $(x_1,x_2)$, as in \cite{PieVer}). Let $k \in \N$, and in each plane $(x_{2i-1},x_{2i})$ let us consider a circular sector of opening $\pi/k$. The intersection of these circular sectors with the ball $B_1$ forms a set $D_k \subset \R^N$ which, by means of reflections in each pair of variables, generates $(2k)^m$ copies covering $B_1$ ($2k$ copies for each pair of variables). For $i=1,\dots,m$ let $\theta_i$ denote the angular variable in the plane
$(x_{2i-1},x_{2i})$, and set $\beta:=\frac{\pi}{2k}$, so that
\[
  D_k=B_1\cap\bigcap_{i=1}^{m}\{0<\theta_i<2\beta\} .
\]
Since $2\beta\le\pi/2$ for $k\ge2$, each sector is the intersection of the two
half-spaces $\{x_{2i}>0\}$ and $\{x_{2i}\cos2\beta-x_{2i-1}\sin2\beta<0\}$; hence $D_k$
is convex, being the intersection of $B_1$ with $2m$ half-spaces. Let $x^k$ be the point
defined by
\[
  (x^k_{2i-1},x^k_{2i})=\frac{d_k}{\sqrt m}\,(\cos\beta,\sin\beta)
  \quad (i=1,\dots,m),\qquad x^k_j=0 \quad (j>2m),
  \qquad d_k:=\frac{\sqrt m}{\sqrt m+\sin\beta},
\]
so that $|x^k|=d_k$ and $\theta_i(x^k)=\beta\in(0,2\beta)$ for every $i$. The distance of
$x^k$ from $\{x_{2i}=0\}$ equals $|x^k_{2i}|=\frac{d_k}{\sqrt m}\sin\beta$, while its
distance from $\{x_{2i}\cos2\beta-x_{2i-1}\sin2\beta=0\}$, whose unit normal is
$(-\sin2\beta,\cos2\beta)$ in the plane $(x_{2i-1},x_{2i})$, equals
\[
  \frac{d_k}{\sqrt m}\,\big|\sin\beta\cos2\beta-\cos\beta\sin2\beta\big|
  =\frac{d_k}{\sqrt m}\,\sin\beta .
\]
By the choice of $d_k$ we also have $\operatorname{dist}(x^k,\partial B_1)=1-d_k
=\frac{d_k}{\sqrt m}\sin\beta$; since $D_k$ is convex, it therefore contains the ball
centred at $x^k$ of radius
\[
  \rho_k=\frac{\sin\beta}{\sqrt m+\sin\beta}
  \ \ge\ \frac{1}{\sqrt m+1}\,\sin\frac{\pi}{2k}
  \ \ge\ \frac{1}{(\sqrt m+1)\,k},
\]
the last inequality following from the concavity of the sine on $[0,\pi]$. Since
$m\le N/2$, the monotonicity of $\Lambda_1$ with respect to domain inclusion and its
scaling properties give $\Lambda_1(D_k)\le\rho_k^{-2}\Lambda_1(B_1)\le Ck^2$, with
$C=C(N)$ independent of $k$ (here and in what follows, $\Lambda_1$ denotes the first eigenvalue of the Dirichlet Laplacian).

Let us consider now
\[
\begin{cases}
-\Delta u+\lambda u = |u|^{p-1}u & \text{in $D_k$} \\
u=0 & \text{on $\pa D_k$}, 
\end{cases} \quad \text{with} \quad \int_{D_k} u^2\,dx = \alpha.
\]
According to \cite[Theorem 1.8]{PieVer}, this problem has a positive solution (obtained as local minimizer of the associated energy on the $L^2$-sphere) for every $\alpha \in (0,\bar \alpha)$, with $\bar \alpha\ge C_{N,p}\Lambda_1(D_k)^{\frac{2}{p-1}-\frac{N}{2}}$. By odd independent reflections in each pair of variables $(x_1,x_2)$, \dots, $(x_{2m-1},x_{2m})$, this gives a non-radial solution to \eqref{pb ex ball}-\eqref{mass constraint} for every $\rho \in (0,\bar \rho)$, with 
\[
\bar \rho \ge (2k)^m C_{N,p}\Lambda_1(D_k)^{\frac{2}{p-1}-\frac{N}{2}} \ge 2^m C_{N,p}' k^{m+\frac{4}{p-1}-N}.
\]
where we used the estimate on $\Lambda_1(D_k)$ and the fact that the exponent of $\Lambda_1(D_k)$ is nonpositive for $p \ge 1+4/N$. Note also that the exponent of $k$ on the right hand side is positive for 
\[
m+\frac{4}{p-1}-N>0 \quad \iff \quad p<1+\frac{4}{N-m}.
\]
For any such exponent, by taking the limit as $k \to +\infty$, the previous argument gives the existence of a solution to \eqref{pb ex ball}-\eqref{mass constraint} for every $\rho>0$. The range $1+4/N \le p<1+4/(N-m)$ covers the full interval $[1+4/N, 2^*-1=1+4/(N-2))$ if and only if $m \ge 2$, namely if and only if $N \ge 4$. This completes the proof.

\section*{Acknowledgments} 
The author is a member of the Gruppo Nazionale per l'Analisi Matematica, la Probabilit\`a e le
loro Applicazioni (GNAMPA) of the Istituto Nazionale di Alta Matematica (INdAM). 
\medskip

\noindent \textbf{Data availability:} No data were used for the research described in the article.

\medskip

\noindent \textbf{Conflict of interest:} The author declares that he has no conflict of interest.

\end{document}